\documentclass[12]{amsart}
\usepackage{amsmath,amssymb,amsthm,color,enumerate,comment,centernot,enumitem,url}

\newtheorem{theorem}{Theorem}[section]
\newtheorem{lemma}[theorem]{Lemma}

\newtheorem{proposition}[theorem]{Proposition}

\newtheorem{question}[theorem]{Question}
\theoremstyle{remark}
    \newtheorem*{remark}{{\bf Remark}}

\theoremstyle{definition}

\numberwithin{equation}{section}

\def\Z {{\mathbb Z}}
\def\ZZ {{\mathcal Z}}

\def\B {{\mathcal B}}

\def\S {{{\mathcal S}}}
\def\A {{\mathcal A}}

\def\red#1 {\textcolor{red}{#1 }}
\def\blue#1 {\textcolor{blue}{#1 }}

\numberwithin{equation}{section}

\def\Z {{\mathbb Z}}
\begin{document}

\title[Sum of Two Squares Modulo $n$]{Representing Integers as the Sum of Two Squares in the Ring $\Z_n$}

\author{Joshua Harrington}
\address{Department of Mathematics, Shippensburg University, Pennsylvania, USA}
\email[Joshua~Harrington]{JSHarrington@ship.edu}

\author{Lenny Jones}
\address{Department of Mathematics, Shippensburg University, Pennsylvania, USA}
\email[Lenny~Jones]{lkjone@ship.edu}

\author{Alicia Lamarche}
\address{Department of Mathematics, Shippensburg University, Pennsylvania, USA}
\email[Alicia~Lamarche]{al5903@ship.edu}

\date{\today}

\begin{abstract}
A classical theorem in number theory due to Euler states that a positive integer $z$ can be written as the sum of two squares if and only if all prime factors $q$ of $z$, with $q\equiv 3 \pmod{4}$, have even exponent in the prime factorization of $z$. One can consider a minor variation of this theorem by not allowing the use of zero as a summand in the representation of $z$ as the sum of two squares. Viewing each of these questions in $\Z_n$, the ring of integers modulo $n$, we give a characterization of all integers $n\ge 2$ such that every $z\in \Z_n$ can be written as the sum of two squares in $\Z_n$.
\end{abstract}
\maketitle
\section{Introduction}\label{Section:Intro}
We begin with a classical theorem in number theory due to Euler \cite{S}.
\begin{theorem}\label{Thm:Sumof2squares}
 A positive integer $z$ can be written as the sum of two squares if and only if all prime factors $q$ of $z$ with $q\equiv 3 \pmod{4}$ have even exponent in the prime factorization of $z$.
\end{theorem}
Euler's complete proof of Theorem \ref{Thm:Sumof2squares} first appeared in a letter to Goldbach  \cite{S}, dated April 12, 1749. His proof uses a technique known as the \emph{method of descent} \cite{HW}, which was first used by Fermat to show the nonexistence of nontrivial solutions to certain Diophantine equations.
Note that, according to Theorem \ref{Thm:Sumof2squares}, the positive integer 9, for example, can be written as the sum of two squares. Since there is only way to write 9 as the sum of two squares, namely $9=3^2+0^2$, we conclude that $0^2$ is allowed as a summand in the representation as the sum of two squares for the integers described in Theorem \ref{Thm:Sumof2squares}. So, a somewhat natural question to ask is the following.
 \begin{question}\label{Q1}
  What positive integers $z$ can be written as the sum of two nonzero squares?
 \end{question}
 The authors could not find a reference for the answer to Question \ref{Q1} in the literature, and although it is not our main concern in this article, we nevertheless provide an answer for the sake of completeness. The following well-known result is due originally to Diophantus \cite{HW}.
 \begin{lemma}\label{Lem:Closed}
   The set of positive integers that can be written as the sum of two squares is closed under multiplication.
 \end{lemma}
 \begin{proof}
 Let $z_1$ and $z_2$ be positive integers such that $z_1=a^2+b^2$ and $z_2=c^2+d^2$. Then
 \[z_1z_2=(ac-bd)^2+(ad+bc)^2. \qedhere\]
 \end{proof}
 Lemma \ref{Lem:Closed} allows us to establish the following partial answer to Question \ref{Q1}.
 \begin{proposition}\label{Prop:EulerExt}
   Let $p\equiv 1 \pmod{4}$ be a prime, and let $a$ be a positive integer. Then there exist nonzero squares $x^2$ and $y^2$ such that $p^a=x^2+y^2$.
 \end{proposition}
 \begin{proof}
 Using Theorem \ref{Thm:Sumof2squares} and the fact that $p$ is prime, we have that $p=c^2+d^2$ for some integers $d>c>0$. If $a$ is odd, then $p^a$ is not a square, and the theorem follows from Lemma \ref{Lem:Closed}. If $a$ is even, then we can write $p^a=p^2p^{a-2}$ with $a-2$ even. Since
  \[p^2=\left(d^2-c^2\right)^2+\left(2cd\right)^2,\]
  we have that
  \[p^a=\left(p^{(a-2)/2}\left(d^2-c^2\right)\right)^2+\left(p^{(a-2)/2}\left(2cd\right)\right)^2,\] with neither summand equal to zero.
 \end{proof}

 To provide a complete answer to Question \ref{Q1}, we let $\ZZ$ denote the set of all integers described in Theorem \ref{Thm:Sumof2squares}, and we ask the following, somewhat convoluted, question.
  \begin{question} \label{Q2}
 Which integers $z\in \ZZ$ actually do require the use of zero when written as the sum of two squares?
  \end{question} Certainly, the integers $z$ that answer Question \ref{Q2} are squares themselves, and therefore we have that $z=c^2$, for some positive integer $c$, and no integers $a>0$ and $b>0$ exist with $z=c^2=a^2+b^2$. In other words, $\sqrt{z}$ is not the third entry in a Pythagorean triple $(a,b,c)$.
 Pythagorean triples $(a,b,c)$ can be described precisely in the following way.
 \begin{theorem}\label{Thm:Py}
 The triple $(a,b,c)$ is a Pythagorean triple if and only if there exist integers $k>0$ and $u>v>0$ of opposite parity with $\gcd(u,v)=1$, such that
 \[a=(u^2-v^2)k, \quad b=(2uv)k \quad \mbox{and} \quad c=(u^2+v^2)k.\]
 \end{theorem}

 Thus, we have the following.

 \begin{theorem}\label{Thm:Nonzero2}
 Let $\widehat{\ZZ}$ be the set of positive integers that can be written as the sum of two nonzero squares. Then, $z\in \widehat{\ZZ}$
if and only if $z\in \ZZ$, and if $z$ is a perfect square, then $\sqrt{z}=(u^2+v^2)k$ for some integers $k>0$ and $u>v>0$ of opposite parity with $\gcd(u,v)=1$.
 \end{theorem}

However, a closer look reveals a somewhat more satisfying description for the integers $z\in \widehat{\ZZ}$ in Theorem \ref{Thm:Nonzero2}.
\begin{theorem}\label{Thm:Nonzero3}
  Let $\widehat{\ZZ}$ be the set of positive integers that can be written as the sum of two nonzero squares. Then, $z\in \widehat{\ZZ}$
if and only if all prime factors $p$ of $z$ with $q\equiv 3 \pmod{4}$ have even exponent in the prime factorization of $z$, and if $z$ is a perfect square, then $z$ must be divisible by some prime $p\equiv 1 \pmod{4}$.
\end{theorem}
\begin{proof}
Suppose first that $z\in\widehat{\ZZ}$. Then $z\in \ZZ$ and all prime factors $q$ of $z$ with $q\equiv 3 \pmod{4}$ have even exponent in the prime factorization of $z$ by Theorem \ref{Thm:Sumof2squares}. So, suppose that $z=c^2$ for some positive integer $c$, and assume, by way of contradiction, that $z$ is divisible by no prime $p\equiv 1 \pmod{4}$. By Theorem \ref{Thm:Nonzero2}, we can write $c=\left(u^2+v^2\right)k$ for some integers $k>0$ and $u>v>0$ of opposite parity with $\gcd(u,v)=1$. Since no prime $p\equiv 1 \pmod{4}$ divides $z$, we have that no prime $p\equiv 1 \pmod{4}$ divides $u^2+v^2$. Note that $u^2+v^2$ is odd, and so every prime $q$ dividing $u^2+v^2$ is such that $q\equiv 3 \pmod{4}$. Thus, by Theorem \ref{Thm:Sumof2squares}, every prime divisor of $u^2+v^2$ has even exponent in the prime factorization of $u^2+v^2$. In other words, $u^2+v^2$ is a perfect square. Hence, $u^2+v^2\in \widehat{\ZZ}$, and by Theorem \ref{Thm:Nonzero2}, we have that
\[\sqrt{u^2+v^2}=\left(u_1^2+v_1^2\right)k_1,\]
for some integers $k_1>0$ and $u_1>v_1>0$ of opposite parity with $\gcd(u_1,v_1)=1$. We can repeat this process, but eventually we reach an integer that is the sum of two distinct squares that has a prime factor $q\equiv 3 \pmod{4}$ that occurs to an odd power in its prime factorization. This contradicts Theorem \ref{Thm:Sumof2squares}, and completes the proof in this direction.

If $z$ is not a perfect square and every prime factor $q$ of $z$ with $q\equiv 3 \pmod{4}$ has even exponent in the prime factorization of $z$, then $z$ can be written as the sum of two squares by Theorem \ref{Thm:Sumof2squares}; and moreover, these squares must be nonzero since $z$ is not a square itself. Thus, $z\in \widehat{\ZZ}$ in this case. Now suppose that $z$ is a perfect square and $z$ is divisible by some prime $p\equiv 1 \pmod{4}$. Let $z=p^{2e}\prod_{i=1}^t\left(r_i\right)^{2e_i}$ be the canonical factorization of $z$ into distinct prime powers. By Proposition \ref{Prop:EulerExt}, there exist integers $u>v>0$, such that $p^{2e}=u^2+v^2$. Then
\[z=\left(u\prod_{i=1}^t\left(r_i\right)^{e_i}\right)^2
+\left(v\prod_{i=1}^t\left(r_i\right)^{e_i}\right)^2\in \widehat{\ZZ},\]
and the proof is complete.
\end{proof}
\begin{remark}
  The method of proof used to establish the first half of Theorem \ref{Thm:Nonzero3} is reminiscent of Fermat's method of descent \cite{HW}.
\end{remark}
In this article, we move the setting from $\Z$ to $\Z_n$, the ring of integers modulo $n$, and we investigate a modification of Question \ref{Q1} in this new realm. In particular, we discover for certain values of $n$ that every element in $\Z_n$ can be written as the sum of two nonzero squares. It is our main goal to characterize, in a precise manner, these particular values of $n$. For the sake of completeness, we also characterize those values of $n$ such that every $z\in \Z_n$ can be written as the sum of two squares where the use of zero is allowed as a summand in such a representation of $z$.

\section{Preliminaries and Notation}\label{Prelim}
To establish our results, we need some additional facts that follow easily from well-known theorems in number theory. We state these facts without proof.
The first proposition follows immediately from the Chinese remainder theorem, while the second proposition is a direct consequence of Hensel's lemma.
\begin{proposition}\label{Prop:CRT}\cite{IR}
Suppose that $m_1,m_2,\ldots ,m_t$ are integers with $m_i\ge 2$ for all $i$, and $\gcd(m_i,m_j)=1$ for all $i\ne j$. Let $c_1,c_2,\ldots,c_t$ be any integers, and let $x\equiv c \pmod{M}$ be the solution of the system of congruences $x\equiv c_i \pmod{m_i}$ using the Chinese remainder theorem. Then there exists $y$ such that $y^2\equiv c \pmod{M}$ if and only if there exist $y_1,y_2,\ldots ,y_t$ such that $y_i^2\equiv c_i \pmod{m_i}$.
\end{proposition}

\begin{proposition}\label{Prop:HL}\cite{N}
Let $p$ be a prime, and let $z$ be an integer. If there exists $x$ such that $x^2\equiv z \pmod{p}$, then there exists $x_k$ such that $\left(x_k\right)^2\equiv z \pmod{p^k}$ for every integer $k\ge 2$.
\end{proposition}

 Throughout this article, we let $\left(\frac{x}{p}\right)$ denote the Legendre symbol, where $p$ is a prime and $x\in \Z$.
  Given an integer $n\ge2$, we let
  \[\S_n=\left\{x^2 \pmod{n} \left\vert\vphantom{\frac{1}{1}}\right. x^2 \not \equiv 0 \pmod{n} \right\},\]
   and \[\S_n^{0}:=\S_n \cup \{0 \pmod{n}\}.\] Then, for a given $z\in \Z_n$, a pair $(x^2,y^2)$ such that
  \begin{equation}\label{Eq:zmodn}
     x^2+y^2\equiv z \pmod{n}
     \end{equation}
   where both $x^2 \pmod{n}$ and $y^2\pmod{n}$ are elements of $\S_n$, is called a \emph{nontrivial solution} to \eqref{Eq:zmodn}. A solution $(x^2,y^2)$ to \eqref{Eq:zmodn}, where either $x^2\equiv 0 \pmod{n}$ or $y^2\equiv 0 \pmod{n}$, is called a \emph{trivial solution}.

  \section{Not Allowing Zero as a Summand}\label{Section:NoZero}
In this section we prove the main result in this article, but first we prove a lemma.
\begin{lemma}\label{Lem:SolvableCong}
 Let $z$ and $a\ge 1$ be integers. Let $p\equiv 1 \mod{4}$ and $q\equiv 3\pmod{4}$ be primes. Then each of the congruences
\begin{align}
  x^2+y^2&\equiv z \pmod{2}\label{C1}\\
  x^2+y^2&\equiv z \pmod{p^{a}}\label{C2}\\
  x^2+y^2&\equiv z \pmod{q}.\label{C3}
\end{align} has a solution. Moreover, with the single exception of $z\equiv  0\pmod{q}$, we can choose a solution where either $x^2 \not\equiv  0\pmod{m}$ or $y^2\not \equiv 0 \pmod{m}$ with $m\in \{2,p^a,q\}$.
\end{lemma}
\begin{proof}
   Clearly, \eqref{C1} always has a solution with $x^2\equiv 1 \pmod{2}$. We show now that \eqref{C2} always has a solution with $y^2\not \equiv 0\pmod{p^a}$. Suppose first that $z\equiv 0 \pmod{p^a}$. By Proposition \ref{Prop:EulerExt}, there exist positive integers $x^2$ and $y^2$ such that $x^2+y^2=p^a$. Then, since neither $x^2$ nor $y^2$ is divisible by $p^a$, we have a desired solution to \eqref{C2}. Now suppose that $z\not \equiv 0\pmod{p^a}$. Let $\gcd(z,p^a)=p^b$ with $b<a$, and write $z=z^{\prime}p^b$. Consider the arithmetic progression
\[\A_k:=4p^{a-b}k + p^{a-b}(1-z^{\prime})+z^{\prime}.\]
Note that for any integer $k$, we have that  $\A_k\equiv z^{\prime} \pmod{p^a}$ and $\A_k\equiv 1 \pmod{4}$. Then, since $\gcd\left(4p^{a-b},p^{a-b}(1-z^{\prime})+z^{\prime}\right)=1$, it follows from Dirichlet's theorem on primes in an arithmetic progression that $\A_k$ contains infinitely many primes $r\equiv 1 \pmod{4}$. For such a prime $r$, Theorem \ref{Thm:Nonzero3} tells us that there exist nonzero integers $x^2$ and $y^2$ such that $x^2+y^2=p^br$. Observe that $x^2$ and $y^2$ cannot both be divisible by $p^a$. Hence, since $p^br\equiv z \pmod{p^a}$, we have a solution to \eqref{C2}, where, after relabeling if necessary, $y^2\not \equiv 0\pmod{p^a}$.

We show next that $\eqref{C3}$ always has a solution. If $z\equiv 0 \pmod{q}$, then we can take $x^2\equiv y^2\equiv 0 \pmod{q}$. If $z\not \equiv  0\pmod{q}$, then we consider the arithmetic progression
\[\B_k:=4qk+q(3+z)+z.\] Note here that  $\B_k\equiv z \pmod{q}$ and $\B_k\equiv 1 \pmod{4}$ for any integer $k$. As before, since $\gcd\left(4q,q(3+z)+z\right)=1$, it follows from Dirichlet's theorem that $\B_k$ contains infinitely many primes $r\equiv 1 \pmod{4}$. Thus, by Proposition \ref{Prop:EulerExt}, there exist nonzero integers $x^2$ and $y^2$ such that $x^2+y^2=r$ for such a prime $r$. Clearly, not both $x^2$ and $y^2$ are divisible by $q$. Hence, with the exception of $z\equiv 0 \pmod{q}$, we have a solution to \eqref{C3} where we can choose $y^2\not \equiv  0\pmod{q}$.
\end{proof}
 \begin{theorem}\label{Thm:NoZero}
 Let $n\ge 2$ be an integer. Then, for every $z\in \Z_n$, \eqref{Eq:zmodn} has a nontrivial solution if and only if
  \begin{enumerate}[font=\normalfont]
    \item $n \not \equiv 0 \pmod{q^2}$ for any prime $q\equiv 3 \pmod{4}$ with $n\equiv 0 \pmod{q}$\label{Thm:Item1}
    \item $n \not \equiv 0 \pmod{4}$ \label{Thm:Item2}
    \item $n\equiv 0 \pmod{p}$ for some prime $p\equiv 1 \pmod{4}$ \label{Thm:Item3}
    \item Also, when $n\equiv 1 \pmod{2}$, we have the following additional conditions. Write $n=5^km$, where $m\not \equiv 0 \pmod{5}$. Then, either \label{Thm:Item4}
          \begin{enumerate}[font=\normalfont]
      \item $k\ge 3$, with no further restrictions on $m$, or \label{Thm:Item5}
      \item $k<3$ and $m\equiv 0 \pmod{p}$ for some prime $p\equiv 1 \pmod{4}$.\label{Thm:Item6}
      \end{enumerate}
      \end{enumerate}
\end{theorem}

\begin{proof}
  Suppose first that, for every $z\in \Z_n$, \eqref{Eq:zmodn} has a nontrivial solution. Let $q$ be a prime divisor of $n$. Then there exist $a^2,b^2,c^2,d^2,e^2,f^2\in \S_n$ such that
  \begin{align}
   a^2+b^2&\equiv q \pmod{n}, \label{Item1}\\
   c^2+d^2&\equiv -1 \pmod{n}\quad \mbox{and} \label{Item2}\\
   e^2+f^2&\equiv 0 \pmod{n} \label{Item3}.
  \end{align}

  Suppose that $q\equiv 3 \pmod{4}$ is a prime such that $n\equiv 0 \pmod{q^2}$. Then we have from \eqref{Item1} that
  \begin{equation}\label{Eq1}
  a^2 + b^2 = kq^2 +q=q(kq+1),
  \end{equation}
  for some nonzero $k\in \Z$. However, \eqref{Eq1} contradicts Theorem \ref{Thm:Sumof2squares}, since clearly $q$ divides $q(kq+1)$ to an odd power. This proves that \eqref{Thm:Item1} holds.

   If $n\equiv 0 \pmod{4}$, then we have from \eqref{Item2} that $c^2+d^2\equiv 3 \pmod{4}$, which is impossible since the set of all squares modulo 4 is $\{0,1\}$. Hence, \eqref{Thm:Item2} holds.

   We see from \eqref{Item3} that $e^2\equiv -f^2 \pmod{q}$ for every prime $q\equiv 3 \pmod{4}$ with $n\equiv 0 \pmod{q}$. Since $\left(\frac{-1}{q}\right)=-1$  for primes $q\equiv 3 \pmod{4}$, we deduce that $e\equiv f \equiv 0\pmod{q}$. Hence, if $n\equiv 1 \pmod{2}$ and $n$ is divisible by no prime $p\equiv 1 \pmod{4}$, it follows from \eqref{Thm:Item1} that $e\equiv f\equiv 0 \pmod{n}$, which contradicts the fact that $e^2,f^2\in \S_n$. From \eqref{Thm:Item2}, if $n\equiv 0 \pmod{2}$, then we can write $n=2m$, where $m \equiv 1 \pmod{2}$. By hypothesis, there exist $s^2,t^2\in \S_n$ such that $s^2+t^2\equiv m \pmod{n}$. If $m$ is divisible by no prime $p\equiv 1 \pmod{4}$, then as before, since $\left(\frac{-1}{q}\right)=-1$ for primes $q\equiv 3 \pmod{4}$, we conclude that $s\equiv t \equiv 0\pmod{m}$. But $s^2+t^2\equiv 1 \pmod{2}$ which implies, without loss of generality, that $s\equiv 0 \pmod{2}$. Therefore, $s\equiv 0\pmod{n}$, which contradicts the fact that $s^2\in \S_n$. Thus, \eqref{Thm:Item3} holds.

    Assume now that $n\equiv 1 \pmod{2}$, and write $n=5^km$, where $m\not \equiv 0 \pmod{5}$. Consider first the possibility that $k=1$ and no prime $p\equiv 1 \pmod{4}$ divides $m$. To rule this case out, we assume first that $\left(\frac{m}{5}\right)=1$. By hypothesis, there exist $s^2,t^2\in \S_n$ such that $s^2+t^2\equiv m \pmod{n}$. If $m=1$, then $n=5$ and this is impossible since the set of nonzero squares modulo 5 is $\{1,4\}$. If $m>1$ then every prime divisor $q$ of $m$ is such that $q\equiv 3 \pmod{4}$. So, we must have, as before, that $s\equiv t \equiv 0 \pmod{m}$. Therefore, since $s^2,t^2\in \S_n$, we deduce that $s^2\not \equiv 0 \pmod{5}$ and $t^2\not \equiv  0\pmod{5}$. Since $\left(\frac{m}{5}\right)=1$, it follows modulo 5 that $s^2, t^2, m\in \{1,4\}$. But then again, $s^2+t^2\equiv m \pmod{5}$ is impossible. If $\left(\frac{m}{5}\right)=-1$, then the proof is identical, except that the representation $s^2+t^2\equiv 2m \pmod{n}$ is impossible since modulo 5 we have $m\in \{2,3\}$, which implies that $s^2+t^2\equiv 2m \pmod{5}$ is impossible.

    The possibility that $k=2$ and no prime $p\equiv 1 \pmod{4}$ divides $m$ can be ruled out in a similar manner by using the fact that the nonzero squares modulo 25 are \{1, 4, 6, 9, 11, 14, 16, 19, 21, 24\}, and reducing the situation to an examination of the representations
    \[s^2+t^2\equiv \left\{\begin{array}{cl}
      1 \pmod{25} & \mbox{if $m=1$}\\
      m \pmod{25} & \mbox{if $m>1$ and $\left(\frac{m}{5}\right)=1$}\\
      2m \pmod{25} & \mbox{if $m>1$ and $\left(\frac{m}{5}\right)=-1$}.
    \end{array}\right.\]
    This completes the proof of the theorem in this direction.

    Now suppose that conditions \eqref{Thm:Item1}, \eqref{Thm:Item2}, \eqref{Thm:Item3} and \eqref{Thm:Item4} hold, and let $z$ be a nonnegative integer. Our strategy here is to use Lemma \ref{Lem:SolvableCong} and Proposition \ref{Prop:CRT} to piece together the solutions for each distinct prime power dividing $n$ to get a nontrivial solution to \eqref{Eq:zmodn}.

   We consider two cases: $n\equiv 0\pmod{2}$ and $n\equiv 1  \pmod{2}$. If $n\equiv 0\pmod{2}$, then we can write
\[n=2\left(\prod_{i=1}^{s} p_i^{a_i}\right)\prod_{i=1}^t q_i,\]
where $s\ge 1$, $t\ge 0$, $p_i\equiv 1 \pmod{4}$ and $q_i\equiv 3 \pmod{4}$. Note that $t=0$ is a possibility, and in this case, we define the empty product $\prod_{i=1}^t q_i$ to be 1. Since $s\ge 1$, we have from Lemma \ref{Lem:SolvableCong} that there exist solutions to \eqref{C1} and \eqref{C2} where respectively $x^2\not \equiv 0\pmod{2}$ and $y^2\not \equiv  0\pmod{p_i^{a_i}}$. Then, using Proposition \ref{Prop:CRT} to piece together the solutions for $x^2$ and $y^2$ modulo each modulus in $\left\{2,p_1^{a_1},\ldots, p_t^{a_t}\right\}$, we get a nontrivial solution to \eqref{Eq:zmodn}.

  We now turn our attention to the case $n\equiv 1 \pmod{2}$, and write
  \[n=5^k\left(\mathop{\prod_{i=1}^s}_{p_i\ne 5} p_i^{a_i}\right)\prod_{i=1}^t q_i.\]
  where $p_i\equiv 1 \pmod{4}$ and $q_i\equiv 3 \pmod{4}$ are primes. Suppose first that $k\le 2$. Then, it is easy to check that the only solutions to
  \[x^2+y^2\equiv 1 \pmod{5^k}\] have either $x^2\equiv 0 \pmod{5^k}$ or $y^2\equiv 0 \pmod{5^k}$. However, we can always choose a solution with $x^2\not \equiv 0 \pmod{5^k}$. Since $k\le 2$, we have that $s\ge 1$ so that there exists a prime $p\equiv 1 \pmod{4}$ that divides $n$, with $p\ne 5$. Hence, by Lemma \ref{Lem:SolvableCong}, \eqref{C2} always has a solution where $y^2\not \equiv 0\pmod{p^a}$. This allows us again to use Proposition \ref{Prop:CRT} to get a nontrivial solution to \eqref{Eq:zmodn}.

  Suppose next that $k\ge 3$. If $s\ne 0$, then as before, we can invoke Lemma \ref{Lem:SolvableCong} and Proposition \ref{Prop:CRT} to achieve a nontrivial solution to \eqref{Eq:zmodn}. So, assume that $s=0$. We show that the congruence
  \begin{equation}\label{Eq:5^k}
    x^2+y^2\equiv z \pmod{5^k},
  \end{equation}
  always has a solution where $x^2\not \equiv 0\pmod{5^k}$ and $y^2\not \equiv 0 \pmod{5^k}$. Since $x^2+y^2=5^k$ has a solution (by Theorem \ref{Thm:Nonzero3}) with neither $x^2$ nor $y^2$ divisible by $5^k$, it follows that \eqref{Eq:5^k} has a nontrivial solution when $z\equiv  0\pmod{5^k}$. Now suppose that $z\not \equiv 0 \pmod{5^k}$. We know from Lemma \ref{Lem:SolvableCong} that \eqref{Eq:5^k} has a solution with $y^2\not \equiv 0 \pmod{5^k}$. If $z\not \in \S_{5^k}$, then it must be that $x^2\not \equiv 0 \pmod{5^k}$ as well, which gives us a nontrivial solution. So, let $z\in \S_{5^k}$. Since $-24\equiv 1\in \S_5$, it follows from Proposition \ref{Prop:HL} that, for any integer $k\ge 2$, there exists $x$ such that
  \begin{equation}\label{Eq:f}
  x^2\equiv -24 \pmod{5^k},
  \end{equation} with $x^2\not \equiv 0 \pmod{5^k}$. We can rewrite \eqref{Eq:f} as
  \begin{equation}\label{Eq:z=1}
  x^2+5^2\equiv 1 \pmod{5^k},
  \end{equation} which implies that \eqref{C2} has a nontrivial solution when $z\equiv 1 \pmod{5^k}$--provided that $k\ge 3$, which we have assumed here. Also, note that this nontrivial solution to \eqref{Eq:z=1} has $x^2\not \equiv 0 \pmod{5}$. Hence, for any $z\in \S_{5^k}$ with $z\not \equiv 0 \pmod{5}$, we see that multiplying \eqref{Eq:z=1} by $z$ yields a nontrivial solution to \eqref{C2} for these particular values of $z$. Now suppose that $z\in \S_{5^k}$ with $z\equiv  0\pmod{5}$. Then $z-1\equiv 4 \pmod{5}$ and, by Proposition \ref{Prop:HL}, we have, for any integer $k\ge 2$, that there exists $x\not \equiv 0 \pmod{5^k}$ such that $x^2\equiv z-1 \pmod{5^k}$. That is,
  \[x^2+1\equiv z \pmod{5^k},\] and hence we have a nontrivial solution to \eqref{Eq:zmodn} in this last case, which completes the proof of the theorem.
 \end{proof}

 The first 25 values of $n$ satisfying the conditions of Theorem \ref{Thm:NoZero} are
\[
 10, 13, 17, 26, 29, 30, 34, 37, 39, 41, 50, 51, 53, 58, 61, 65, 70, 73, 74, 78, 82, 85, 87, 89, 91.
\]

\section{Allowing Zero as a Summand}\label{Section:AllowingZero}
For the sake of completeness, we address now the situation when trivial solutions are allowed in \eqref{Eq:zmodn}.
The main theorem of this section gives a precise description of the integers $n$ such that, for any $z\in \Z_n$, \eqref{Eq:zmodn} has a solution $(x,y)$ with $x,y\in \S_n^{0}$.
 Certainly, the proof of this result builds off of Theorem \ref{Thm:NoZero} since every value of $n$ for which there exists a nontrivial solution to \eqref{Eq:zmodn} will be included here as well. From an analysis of the proof of Theorem \ref{Thm:NoZero}, it is straightforward to see that allowing $0$ as a summand does not buy us any new values of $n$ here under the restrictions found in \eqref{Thm:Item1} and \eqref{Thm:Item2}. However, it turns out that the restrictions in \eqref{Thm:Item3} and \eqref{Thm:Item4} of Theorem \ref{Thm:NoZero} are not required. More precisely, we have:
\begin{theorem}\label{Thm:Zero}
 Let $n\ge 2$ be an integer. Then, for every $z\in \Z_n$, \eqref{Eq:zmodn} has a solution $(x,y)$ with $x,y\in \S_n^{0}$ if and only if the following conditions hold:
  \begin{enumerate}[font=\normalfont]
    \item $n \not \equiv 0 \pmod{q^2}$ for any prime $q\equiv 3 \pmod{4}$ with $n\equiv 0 \pmod{q}$ \label{Thm:Item1A}
    \item $n \not \equiv 0 \pmod{4}$. \label{Thm:Item2A}
      \end{enumerate}
\end{theorem}
\begin{proof}
 We show first that condition \eqref{Thm:Item3} of Theorem \ref{Thm:NoZero} is not required here. Suppose that every prime divisor $p$ of $n$ is such that $p\equiv 3 \pmod{4}$. Certainly, if $z\in \Z_n$ is a square, then \eqref{Eq:zmodn} has a solution $(x,y)$, with $x,y\in \S_n^{0}$; namely $(z,0)$. So, we need to show that \eqref{Eq:zmodn} has a solution $(x,y)$ with $x,y\in \S_n^{0}$ for every nonsquare $z\in \Z_n$. To begin, we claim that \eqref{Eq:zmodn} has a solution modulo $p$ when $z=-1$, which is not a square modulo $p$.
For $a\in \Z_p$, if $\left(\frac{a}{p}\right)=1$ and $\left(\frac{a+1}{p}\right)=-1$, then $\left(\frac{-a-1}{p}\right)=1$. Thus,
$$a + (-a-1)\equiv -1\pmod p.$$
Such an element $a\in \Z_p$ must exist, otherwise all elements of $\Z_p$ would be squares, which is absurd. Now, any nonsquare $z\in \Z_p$ can be written as $-(-z)$, where $\left(\frac{-z}{p}\right)=1$. Therefore, $\left(\frac{-za}{p}\right)=\left(\frac{-z(-a-1)}{p}\right)=1$, and we have that
\[(-za)+(-z)(-a-1)\equiv z \pmod{p}.\] Then we can use Proposition \ref{Prop:HL} to lift this solution modulo $p$ to a solution modulo $p^a$, where $p^a$ is the exact power of $p$ that divides $n$. Finally, we use Proposition \ref{Prop:CRT} to piece together the solutions for each of these prime powers to get a solution modulo $n$.

To see that the restrictions in \eqref{Thm:Item4} are not required here, we note that the restriction that $m$ be divisible by some odd prime $p\equiv 1 \pmod{4}$ is not required by the previous argument. Therefore, to complete the proof of the theorem, it is enough to observe that every element in $\Z_5$ and $\Z_{25}$ can be written as the sum of two elements $x,y\in \S_n^{0}$.
\end{proof}

The first 25 values of $n$ satisfying the conditions of Theorem \ref{Thm:Zero} are
\[
2, 3, 5, 6, 7, 10, 11, 13, 14, 15, 17, 19, 21, 22, 23, 25, 26, 29, 30, 31, 33, 34, 35, 37, 38.
\]

\section{Future Considerations}
 Theorem \ref{Thm:NoZero} and Theorem \ref{Thm:Zero} consider the situation when the entire ring $\Z_n$ can be obtained as the sum of two squares. When this cannot be attained, how badly does it fail; and is there a measure of this failure in terms of $n$? There are certain clues to the answers to these questions in the proof of Theorem \ref{Thm:NoZero}, but we have not pursued the solution in this article.

\section{Acknowledgments}


\begin{thebibliography}{99}

\bibitem{HW} G. H. Hardy and E. M. Wright, \emph{An Introduction to the Theory of Numbers}, fifth edition, Oxford University Press 1979.
\bibitem{IR} K. Ireland and M. Rosen, \emph{A Classical Introduction to Modern Number Theory}, Springer-Verlag, New York 1990.
\bibitem{N} M. Nathanson, \emph{Elementary Methods in Number Theory}, Springer-Verlag, New York 2000.
\bibitem{S} J. Suzuki, \emph{Euler and Number Theory: A Study in Mathematical Invention}, Leonhard Euler: Life, Work and Legacy, Studies in the History of Philosophy of Mathematics, {\bf 5}, Robert E. Bradley and C. Edward Sandifer, editors, Elsevier B. V., Amsterdam, The Netherlands 2007.






\end{thebibliography}
\end{document}